\documentclass[12pt,reqno]{amsart}
\usepackage{amssymb,amsthm,url}
\usepackage[breaklinks=true]{hyperref}
\usepackage[utf8]{inputenc}
\usepackage[OT2,T1]{fontenc}
\swapnumbers
\let\<\langle
\let\>\rangle

\advance\hoffset by -0.5truein
\DeclareSymbolFont{cyrletters}{OT2}{wncyr}{m}{n}
\DeclareMathSymbol{\Sha}{\mathalpha}{cyrletters}{"58}

\theoremstyle{definition}
\newtheorem{definition}{Definition}[section]
\newtheorem{remark}[definition]{Remark}
\newtheorem{example}[definition]{Example}

\theoremstyle{plain}
\newtheorem{lemma}[definition]{Lemma}
\newtheorem{proposition}[definition]{Proposition}
\newtheorem{corollary}[definition]{Corollary}
\newtheorem{theorem}[definition]{Theorem}

\newenvironment{Proof}[1][Proof.]{\begin{trivlist}
\item[\hskip \labelsep {\bfseries #1}]}{\flushright
$\Box$\end{trivlist}}
\def\Com{\mathrm{Com}}

\usepackage[all]{xy}
\usepackage{pdfsync}
\usepackage{ytableau}

\newcommand{\sudda}[1]{}

\begin{document}

\title{Metabelian Lie and perm algebras}

\author{F. A. Mashurov}

\address{Suleyman Demirel University, Kaskelen, Kazakhstan and  Institute of Mathematics and Mathematical Modeling, Almaty, Kazakhstan}

\email{f.mashurov@gmail.com}

\author{B. K. Sartayev}

\address{Sobolev Institute of Mathematics, Novosibirsk, Russia and Suleyman Demirel University, Kaskelen, Kazakhstan}

\email{baurjai@gmail.com}

\keywords{Perm algebras, Lie and Jordan elements, polynomial identities, Gr\"obner bases}

\subjclass[2010]{17B01, 17B35, 13P10}

\maketitle

\begin{abstract} 

It is well known that any Lie algebra can be embedded into an associative algebra.
We prove that any metabelian Lie algebra can be embedded into an algebra in the subvariety of perm algebras, i.e., associative algebras with the identity
$abc-acb =0$. In addition, a technical method to construct the universal enveloping perm algebra for a metabelian Lie algebra is given.
\end{abstract}

\section{\label{nn}\ Introduction}

Determining a set of identities for an adjoint class of a variety through commutator or anti-commutator products is a popular question in algebra.  This question has been solved or largely answered for some varieties, but it is still open for others. It is well known that a variety of associative algebras is a classic illustration of this approach.
The Poincaré-Birkhoff-Witt's (PBW) theorem provides the entire collection of identities for commutator algebras of associative algebras. Namely,
every identity that holds
on every associative algebra with respect to
the commutator product is a corollary of the anti-commutativity and the Jacobi identity. However, anti-commutator algebras of associative algebras have an entirely different story \cite{McCrimmon}. 

A variety of metabelian Lie algebras is defined by the additional identity  \begin{equation}\label{metabelian id}
[[a,b],[c,d]]=0. \end{equation} This variety  has received a lot of attention and was considered in \cite{DDF13}, \cite{Dual 1}, \cite{Scheuneman}, \cite{CHF}, \cite{Dual 2}, \cite{Romankov}. A basis of free metabelian Lie algebras was given in \cite{Bahturin1973}.

In this paper, we show that any metabelian Lie algebra can be embedded into an algebra in the subvariety of associative algebras defined by the right-commutativity identity:
\begin{equation}\label{perm id}
abc-acb=0. \end{equation}

An associative algebra with the identity $(\ref{perm id})$ is called a \textit{perm} algebra, see \cite{Perm 1}, \cite{Perm 2}. In \cite{PK-BS} perm algebras with derivation were considered.

Let $A^{(-)}= (A,[\cdot,\cdot])$ be a subalgebra of $A$ under the commutator, i.e., the multiplication is defined by the Lie bracket $[a,b]=ab- ba$. An algebra $A^{(+)}= (A,\{,\})$ can be defined in a similar way, i.e., with the anticommutator product  $\{a,b\}=ab+ba$.

Let $\mathcal{P}erm$ denote the variety of perm algebras. Define $\mathcal{P}erm^{(-)}$ and $\mathcal{P}erm^{(+)}$ as the classes of algebras of the form $P^{(-)}$ and $P^{(+)}$, respectively, where $P\in \mathcal{P}erm$. An algebra $B$ is called \textit{$p^{-}$-special} (\textit{$p^{+}$-special}) concerning Lie (\textit{Jordan}) brackets if there exists a perm algebra $P$ such that $B$ is a subalgebra of $P^{(-)}$ (\textit{$P^{(+)})$}, otherwise it is called \textit{$p^{-}$-exceptional} (\textit{$p^{+}$-exceptional}). The algebra $P$ is called a perm envelope for $B$.
We show that the class of $p^{-}$-special algebras coincides with the variety of metabelian Lie algebras. In the case of anti-commutator product, it is shown that the class of $p^{+}$-special algebras is not a variety since a homomorphic image of a $p^{+}$-special algebra may be $p^{+}$-exceptional. We find all identities that define the variety generated by all $p^{+}$-special algebras: apart from commutativity, there are two identities of degree four.

Let $X=\{x_1,x_2,\ldots\}$ be a set and $P(X)$ be  the free perm algebra generated by $X$. A polynomial in $P(X)$ is called Lie (\textit{Jordan}) element of $P(X)$ if it can be expressed by elements of $X$ in terms of commutators (\textit{anticommutators}). For associative algebras, there is a well-known  Dynkin-Specht-Wever criterion for Lie elements \cite{Dynkin}, and for the Jordan elements the question remains open. Also, metabelian Lie algebras can be obtained from bicommutative algebras
by a commutator \cite{Ismailov N}. Criteria for Lie and Jordan elements in a free bicommutative algebra have been recently found in \cite{Ismailov N}.  P.M. Cohn gave a criterion for the Jordan elements  generated by three elements \cite{Cohn}. In the same paper, the author provided a criterion of homomorphic images to be special and showed that the homomorphic image of the free special Jordan algebra generated by three elements can be exceptional. These questions are considered for the different varieties of algebras, see   \cite{DIM19}, \cite{DIS2022}, \cite{PK-BS2021}.

In this study, we present conditions for Lie elements and describe all Jordan elements in a free perm algebra. The principal results of the paper are derived by using Cohn's criterion for speciality of algebras and the obtained criteria. If an algebra $B$ can be embedded into $ A $, then it is natural to find the base of the universal envelope algebra of $A.$ For this reason, we use the technique of Gr\"obner-Shirshov bases for perm algebras which is
a particular case of the technique from \cite{Kolesnikov}. We show a method to find the basis of the universal enveloping perm algebra for a given metabelian Lie algebra.

In the last section, we consider a variety of (non-associative) commutative algebras with two identities of degree 4 that hold on every algebra from $\mathcal{P}erm^{(+)}$. We construct a linear basis in the free algebra of this variety and prove it to be $p^{+}$-special.

We consider all algebras over a field $K$ of characteristic $0$.

\section{Main Results}

Let $X$ be a non-empty set. We first consider the relation between the metabelian Lie algebras and algebras from $\mathcal{P}erm^{(-)}$.
\begin{theorem}\label{Identities in MLie}
The free metabelian Lie algebra $ML(X)$ is $p^-$-special. \end{theorem}

The Dynkin map $D: P(X)\rightarrow P(X)$ is a linear map,
defined on base elements of $P(X)$ as follows
$$x_{i_1}x_{i_2}\cdots x_{i_n}\mapsto[[\cdots [x_{i_1},x_{i_2}],\cdots ],x_{i_n}],$$
where $i_2\leq\ldots\leq i_n$.

We define \textit{head} of $f\in P(X)$ in the following way.
For $x\in X$, set $head(x)=x$. If $u=x_{i_1}x_{i_2}\dots x_{i_n}$
is a basic monomial in $P(X)$, $i_2\le \dots \le i_n$, then 
\[
head(u) = \begin{cases} 0, & i_1\le i_2, \\
u, &\text{otherwise}.
\end{cases}
\]
For a polynomial $f = \sum_i \alpha_iu_i$, expand $head$ by linearity.
For example, if $f=x_1x_2x_3x_4+x_2x_1x_3x_4-x_3x_1x_2x_4+2x_4x_1x_2x_3+x_2$, then 
$$head(f)=x_2x_1x_3x_4-x_3x_1x_2x_4+2x_4x_1x_2x_3+x_2.$$

Let $L(X)$ be a free special metabelian Lie algebra generated by a set $X,$ i.e., the set of all Lie elements.

\begin{theorem}\label{Lie criterion} Let $f\in P(X)$. Then  $f$ is a Lie element if and only if $D(head(f))=f.$
\end{theorem}

\begin{theorem}\label{hom image minus}
	Every homomorphic image of $L(X)$ is $p^-$-special.
\end{theorem}

Consequently, from Theorem $\ref{Identities in MLie}$ and Theorem $\ref{hom image minus}$ we have:

\begin{corollary}
	Every metabelian Lie algebra has a perm enveloping algebra.
\end{corollary}

Next, we consider algebras from $\mathcal{P}erm^{(+)}.$

\begin{theorem}\label{Identities in perm plus} Every perm algebra under anticommutator satisfies the following identities:
\begin{equation}\label{eq21}
\{a, b\} = \{b, a\},
\end{equation} 
\begin{equation}\label{eq22}\{\{a,b\},\{c,d\}\}=\{\{a,d\},\{b,c\}\},\end{equation}
\begin{equation}\label{eq23} 2\langle\{a,b\},c,d\rangle=\langle\{a,b\},d,c\rangle+\langle\{a,c\},b,d\rangle+\langle\{b,c\},a,d\rangle,
\end{equation}
for $a,b,c,d\in P(X)$ and where $\langle a,b,c\rangle=\{\{a,b\},c\}-\{a,\{b,c\}\}.$ \end{theorem}

\begin{theorem}\label{Jordan elements}
	Every element with degree more than 2 in a free  perm algebra is a Jordan element.
\end{theorem}
Let $SJ(X)$  be a free special algebra on $X$ under the anti-commutator product, i.e., subalgebra of $P(X)^{(+)}=(P(X), \{ , \})$ generated by $X.$ The next theorem is an analogue of Cohn’s theorem on exceptional homomorphic image of the special Jordan algebra in three generators \cite{Cohn}. In our case, we have an  exceptional homomorphic image with two generators.

\begin{theorem}\label{hom image plus}
	The algebra  $SJ(\{x,y\})$ has a $p^+$-exceptional homomorphic image.
\end{theorem}

\begin{corollary}
	The class of all $SJ(X)$ algebras is not variety.
\end{corollary}

Let $J(X)$  be a free algebra on $X$ with the identities $(\ref{eq21}),(\ref{eq22})$ and $(\ref{eq23}).$

\begin{theorem}\label{free J(X) special}
	The free algebra $J(X)$  is $p^+$-special.
\end{theorem}

\section{Proof of theorems}
In this section we give proofs of main results.

\begin{lemma}\label{opening brackets}
 The following identity holds in every perm algebra:
\begin{equation}\label{opening brackets1} [[\cdots[[x_{i_1},x_{i_2}],x_{i_3}],\cdots], x_{i_n}]=[x_{i_1},x_{i_2}]x_{i_3}\cdots x_{i_n}, \, \, \, n\geq 2.\end{equation}\end{lemma}

\begin{proof} 
It follows immediately from the identity $x[a,b]=0.$
\end{proof}

An ideal $I$ of the free perm algebra $P(X)$ is called \textit{T-ideal} if $\xi(I)\subseteq I$ for all endomorphisms $\xi$ of $P(X).$ Denote by $T(u)$ the \textit{T-ideal} in $P(X)$ generated by $u\in P(X).$

\begin{proposition}
   Let $P$ be a perm algebra. Then the Lie algebra $(P, [\cdot,\cdot])$ satisfies the  metabelian identity.
\end{proposition}

\begin{proof} It is easy to see that 
	$[[a,b],[c,d]]\in T(a[b,c]).$
	Since the  identity $(\ref{perm id})$ holds in $P(X),$ hence  $(\ref{metabelian id})$ is an identity in $(P, [\cdot,\cdot]).$

\end{proof}

 	

	
	 Let us recall a free base of metabelian Lie algebra on a set $X$ of degree $n$ \cite{Bahturin1973}:
    \begin{equation}\label{base of Meta.Lie}
        [[\cdots[[x_{i_2},x_{i_1}],x_{i_3}],\cdots], x_{i_n}],
    \end{equation}
	where $i_2>i_1\leq i_3\leq\ldots\leq i_n$.
	
	\subsection{Proof of Theorem \ref{Identities in MLie}}
    \begin{proof}
      
  Consider the homomorphism $\phi$ from $ML(X)$ to $P(X)^{(-)}$ such that $\phi(x) = x$ for $x\in  X.$ The kernel of $\phi$ consists of all identities that hold on all perm algebras with respect to the commutator product. Assume $0\neq f \in ker \, \phi.$ Without loss of generality we may assume $X=\{x_1,x_2,\ldots\}$ and $f=f(x_1,\ldots,x_n)$ is multi-linear in $x_1,\ldots,x_n.$ Then by \cite{Bahturin1973} $ML(X)$ is spanned by the elements in the form $(\ref{base of Meta.Lie}).$ Therefore, 
  $$f=\sum \lambda_{i_2,\ldots,i
  _n}[[\cdots[[x_{i_2},x_{1}],x_{i_3}],\cdots], x_{i_n}],$$
  where $1<i_3<\cdots<i_n$ and $\lambda_{i_2,\ldots,i
  _n}\in K.$
  
  By Lemma \ref{opening brackets} the image $\phi(f)$ is nonzero in $P(X),$ provided that at least one of the coefficients $\lambda_{i_2,\ldots,i
  _n}$ is non zero, since the monomials $\lambda_{i_2,\ldots,i
  _n} x_{i_2} x_{1} x_{i_3}\cdots  x_{i_n}$ are linearly independent in $P(X).$
  
   \end{proof}

		\subsection{Proof of Theorem \ref{Lie criterion}}
	 Let $P_n(X)$ and $L_n(X)$ be the $n$-th homogeneous parts of $P(X)$ and  $L(X),$ respectively.
	\begin{proof}
	Let $f\in P_n(X)$ and suppose that 
	$D(head(f))=f$. Since the image of Dynkin map $D$ is the subspace of $L(X),$ $f$ is Lie element.
	
		Conversely, suppose that $f\in L_n(X).$ Since, $L(X)$ is a free special metabelian Lie algebra $f$ has the following form:
		$$f=\sum \lambda_i[[\cdots[[x_{i_2},x_{i_1}],x_{i_3}],\cdots], x_{i_n}], $$
		where $\lambda_i\in K$ and $i_2>i_1\leq i_3\leq\ldots\leq i_n$.
		
		By Lemma \ref{opening brackets}, for  each monomial $\lambda_i[[\cdots[[x_{i_2},x_{i_1}],x_{i_3}],\cdots], x_{i_n}]$  we have
		$$head(\lambda_i[[\cdots[[x_{i_2},x_{i_1}],x_{i_3}],\cdots], x_{i_n}])=\lambda_i x_{i_2}x_{i_1}x_{i_3}\cdots  x_{i_n}.$$
		If we apply the linear mapping $D$  for the monomial of $P(X)$ on the right  hand of the equality, then 
		$$D(\lambda_i x_{i_2}x_{i_1}x_{i_3}\cdots  x_{i_n})=\lambda_i D(x_{i_2}x_{i_1}x_{i_3}\cdots  x_{i_n})=$$		$$\lambda_i[[\cdots[[x_{i_2},x_{i_1}],x_{i_3}],\cdots], x_{i_n}].$$ 
		We obtain $D(head(f))=f$ by applying the Dynkin map $D$ to other basis elements in $head(f).$ 
		\end{proof}


		\subsection{Proof of Theorem 2.3}
		
		\begin{proof}
		Let $I$ be an ideal in $L(X)\subseteq  P(X)^{(-)}.$ Denote by $J$ the ideal of P(X) generated by $I.$ Since $P(X)L(X) = 0$ by $(\ref{perm id})$, we have $J = I + IP(X).$ By Lemma $\ref{opening brackets}$ $gx_1\cdots x_n = [\cdots [g, x_1],\cdots , x_n] \in L(X)$ for every $g \in L(X).$ Hence, $I = J$ and, in particular, $L(X) \cap J = I.$ The latter implies $L/I \subseteq (P(X)/J)^{(-)}$ as desired.
			\end{proof}
			
		\subsection{Proof of Theorem \ref{Identities in perm plus}}
	    \begin{proof}
	    
	   	The identities $(\ref{eq22})$ and $(\ref{eq23})$ can be easily  verified using the following expansions of the elements				\begin{align}
        \{\{a,b\},\{c,d\}\}=2abcd+2bacd+2cabd+2dabc, \\
        \langle\{a,b\},c,d\rangle=-abcd-bacd+2dabc.
        \tag*{\qedhere}
        \end{align}
		\end{proof}

		\subsection{Proof of Theorem \ref{Jordan elements}}
		\begin{proof}

		Let us write a congruence relation $a\equiv 0 $ for $a\in SJ(X)$.
		Then we have to show that 
		\begin{equation}\label{j1}
		x_1x_2\cdots x_n\equiv0,\end{equation}
		for $n>2$.
		We use induction on $n$. If $n=3$, then 
		$$x_1x_2x_3=-\frac{1}{4}\{\{x_1,x_2\},x_3\}+\frac{3}{4}\{\{x_2,x_3\},x_1\}-\frac{1}{4}\{\{x_1,x_3\},x_2\}.$$
		Assume that $(\ref{j1})$ holds for elements with degree less than $n.$ Then by $(\ref{perm id})$, we obtain
		$$0\equiv x_1x_2\cdots x_{n-2}\{x_{n-1},x_{n}\}=2x_1x_2\cdots x_n.$$ 
		Hence
		\[x_1x_2\cdots x_n\equiv0. \qedhere \]
			\end{proof}

		\subsection{Proof of Theorem \ref{hom image plus}}
		\begin{proof}
	
	Let $I$ be the ideal of $SJ(\{x, y\})$ generated by $\{x, y\},x^3,y^2.$ 
    Then it is easy to see that $B = SJ(\{x, y\})/I$ is a $4$-dimensional algebra with a linear basis $x, y, a=x^2 =\frac{1}{2}\{x, x\}, b = \{a, y\}.$ It is essential that $2b = \{\{x, x\}, y\}\notin I$ since all generators of $I$ are homogeneous in $x$ and $y,$ but the only appropriate element of $I$ is $\{x, \{x, y\}\}$ which is not proportional to $b.$ All other anti-commutators of the basic elements are zero. Denote by $J$ the ideal in $P(\{x, y\})$ generated by $I.$ If $B$ was a $p^+$-special algebra then $I = J \cap SJ(\{x, y\})$ by the Cohn’s criterion \cite{Cohn}, but the latter intersection contains $b$ by Theorem $\ref{Jordan elements}$.

		\end{proof}
		
\section{The basis of universal enveloping perm algebras for metabelian Lie algebras}

In this section, we construct the basis of universal enveloping perm algebras for metabelian Lie algebras and consider one example of metabelian Lie algebra. We use the technique of Gr\"obner-Shirshov bases for perm algebras that can be derived from \cite{Kolesnikov}. More explicitly, the method given in \cite{Kolesnikov} allows one to construct Gr\"obner-Shirshov bases for $di$-$Var$ algebras using the theory of Gr\"obner-Shirshov bases of $Var$ algebras.

In our case, the free perm algebra is a free di-commutative algebra, and the theory of Gr\"obner bases for commutative algebras developed in \cite{Buchberger}. So, we have 
$$Perm\langle X\rangle=di\text{-}Com\langle X\rangle\simeq V\subseteq Com\langle X\cup\dot{X}\rangle,$$
where $V$ is a subalgebra of $Com\langle X\cup\dot{X}\rangle$ generated by $\dot{X}$ relative to the operations $\vdash$ and $\dashv$ such that $a\vdash b=a\dot{b}$ and $a\dashv b=\dot{a}b$. 
Practically, $V$ consists of commutative polynomials in $X\cup \dot X$
that are linear in $\dot X$.
Using this scheme we can construct the basis of universal enveloping perm algebras for a metabelian Lie algebras.

Let $L$ be a metabelian Lie algebra and let $X$
be a linear basis of $L$ presented as $X=Y\cup Z$,
where $Y$ is a basis of $L'=[L,L]$ and $Z$ is a complement of $Y$ to a basis of $L$. 
Suppose both $Y$ and $Z$ are linearly ordered and $Y<Z$.
The universal enveloping Perm-algebra $U_{perm}(L)$ 
of $L$ is generated by $X$ relative to 
the following relations:
\[
x_i \dashv x_j - x_j\vdash x_i = [x_i,x_j],\quad x_i,x_j\in X.
\]
According to the general scheme, given in \cite{Kolesnikov}, we should transform these relations into commutative polynomials.
Namely, $U_{perm}(L)$ may be identified with a subalgebra of 
$A=\Com\<X\cup \dot X \>/(S)$, where 
$S$ consists of 
\[
\begin{gathered}
\dot{y}_iy_j-\dot{y}_jy_i=0,\;\;\; y_i>y_j, \\
\dot{z}_iy_j-\dot{y}_jz_i=\dot{[z_i,y_i]}, \\
\dot{z}_iz_j-\dot{z}_jz_i=\dot{[z_i,z_j]},\;\;\;z_i>z_j.   
\end{gathered}
\]
and the same relations with removed dots:
\[
[z_i,y_j]=0, \quad [z_i,z_j]=0,
\]
where $y_j\in Y$ and $z_i,z_j\in Z$.
The second group of relations says
$y=0$ for $y\in Y$ and, therefore, we may assume that $S$ 
consists of 
\[
\dot{y}_j z_i=\dot{[y_j,z_i]},
\quad 
\dot{z}_iz_j-\dot{z}_jz_i=\dot{[z_i,z_j]},\ i>j.
\]
Suppose $\dot Y<\dot Z<Z$.
To construct the linear basis of $U_{perm}(L)$ in algebra $A$ 
it is enough to consider those compositions of relations 
that contain only one dot. The first such composition 
corresponds to the ambiguity 
$\dot{y}z_iz_j$. On the one side
$$\dot{y}z_iz_j=\dot{[y,z_i]}z_j=\dot{[[y,z_i],z_j]},$$
and from the other side
$$\dot{y}z_iz_j=\dot{[y,z_j]}z_i=\dot{[[y,z_j],z_i]}.$$
Note that the commutator relations hold in $L$, and
by the Jacobi and metabelian identities, we obtain 
$$[[y,z_i],z_j]-[[y,z_j],z_i]=[y,[z_1,z_2]]=0.$$
The second composition comes from the ambiguity 
$\dot{z}_i z_jz_k$, where $i>j>k$.
On the one side
$$\dot{z}_iz_jz_k=\dot{z}_jz_iz_k+\dot{[z_i,z_j]}z_k=\dot{z}_kz_iz_j+\dot{[z_j,z_k]}z_i+\dot{[z_i,z_j]}z_k,$$
and from the other side
$$\dot{z}_iz_jz_k=\dot{[z_i,z_k]}z_j+\dot{z}_kz_iz_j.$$
By the Jacobi identity, we obtain
$$[[z_j,z_k],z_i]+[[z_i,z_j],z_k]-[[z_i,z_k],z_j]=0.$$
Since all compositions are trivial, these relations form a Gr\"obner--Shirshov base of $U_{perm}(L)$ algebra, and the linear basis of 
$U_{perm}(L)$ consists of monomials in $A$ which are linear in $\dot X$
and reduced 
modulo $\dot{y}z$ and $\dot{z}_iz_j$, where $i>j$. These monomials are
$$\dot{y},\;\;\; \dot{z}_{i_1}z_{i_2}\cdots z_{i_n},$$
where $i_1\leq i_2\leq\ldots\leq i_n$, $n\ge 1$.

\begin{example}
Let us consider the three dimensional Heisenberg algebra $\mathcal{H}$, i.e., a 3-nilpotent Lie algebra with the following multiplication table:
$$[e_1,e_2]=-[e_2,e_1]=e_3.$$
By obtained result, the basis monomials of $U(\mathcal{H})$ are
$$\dot{e}_3,\;\;\;\dot{e}_2 e_2\cdots e_2$$
and
$$\dot{e}_1 e_1\cdots e_1 e_2\cdots e_2.$$
In terms of perm algebra, basis monomials of $U(\mathcal{H})$ are $e_3$, $e_2^n$, $e_1^ke_2^l$ and $\mathcal{H}\hookrightarrow U(\mathcal{H})$.
\end{example}

\begin{remark}
Note that the pair of varieties ($\mathcal{P}erm$, $\mathcal{ML}ie$) is not a PBW-pair in the sense of \cite{Mikhalev-Shestakov2014}. 
Indeed, the Gel'fand--Kirillov dimension of $U(\mathcal{H})$ equals two, but the universal enveloping  perm algebra
of the 3-dimensional abelian Lie algebra is just the polynomial algebra in 3 variables, so its 
Gel'fand--Kirillov dimension is 3.
\end{remark}

\section{\label{nn1}\ Proof of Theorem \ref{free J(X) special}}

Let $X=\{x_1,x_2,\ldots\}$ and $J(X)$ be the free algebra generated by $X$ in the variety defined by the following identities:
\begin{equation}\label{id21}
a b = b a,
\end{equation} 
\begin{equation}\label{id22} (ab)(c d)=(ad)(bc),\end{equation}
\begin{equation}\label{id23} (ab)(c d)=-2((ab)c)d + (( a b)d)c +(( ac)b)d +(( bc)a)d=0
\end{equation}
for $a,b,c,d\in J(X).$

Note that the identities $(\ref{id21})-(\ref{id23})$ are equivalent to $(\ref{eq21})-(\ref{eq23}).$ Denote a left-normed element $((x_1x_2)\cdots )x_n\in J(X) $ as $x_1x_2\cdots x_n,$ where $x_i\in X, i\in\{1,\ldots,n\}.$

\begin{lemma} Every element in  $J(X)$ can be expressed as a linear combination of left-normed elements.

\end{lemma}

\begin{proof} Let $w\in J(X)$ be a monomial of degree $n.$ For $n=1,2$ it is trivial. For $n=3,$ it follows from $(\ref{id21})$. For $n>3,$ the statement can be easily proved by induction via $(\ref{id23}).$

\end{proof}

For $a,b,c\in J(X),$ let us define $$f(a;b,c)=-\frac{1}{4}\big((ab)c-3(bc)a+(ac)b\big).$$

Note that from the above definition we have 
\begin{equation}\label{newident40}
    f(a;b,c)=f(a;c,b).
\end{equation}

\begin{proposition} The following identities hold in $J(X)$

\begin{equation}\label{newident41}
    (ab)c=f(a;b,c)+f(b;a,c)+2f(c;a,b)
\end{equation}

\begin{equation}\label{newident43}
    f(a;b,cd)=f(a;bc,d),
\end{equation}

\begin{equation}\label{newident44}
    f(a;b,c)d=\frac{1}{2}f(a;b,cd)+\frac{1}{2}f(d;a,bc),
\end{equation}

\begin{equation}\label{newident431}
     f(a;b,(cd)e)=f(a;b,c(de)),
\end{equation}

\end{proposition}

\begin{proof} The identity $(\ref{newident41})$ follows from commutative identity and straightforward calculations.
Let $G(X)$ be the free commutative algebra with identity $(\ref{id22})$ generated by $X.$ Define $$\psi(a,b,c,d)=-(ab)(c d)-2((ab)c)d + (( a b)d)c +(( ac)b)d +(( bc)a)d,$$
for $a,b,c,d\in G(X).$ Since $\psi(a,b,c,d)=0$ in $J(X),$ it is enough to show that the identities $(\ref{newident43})-(\ref{newident431})$ are linear combinations of polynomials of the form $\psi(x,y,z,t)w$ in $G(X),$ where $x,y,z,t,w\in G(X).$ 

Then the identity $(\ref{newident43})$ follows from $$f(a;b,cd)-f(a;bc,d)=\frac{1}{4}\psi(c,b,d,a)-\frac{1}{4}\psi(c,d,b,a).$$

The identity $(\ref{newident44})$ follows from

$$f(a;b,c)d-\frac{1}{2}f(a;b,cd)-\frac{1}{2}f(d;a,bc)=$$
$$\frac{3}{32} \psi(b,a,c,d)-\frac{5}{32} \psi(b,a,d,c)+\frac{5}{32} \psi(c,a,b,d)+$$$$\frac{1}{32} \psi(c,a,d,b)-\frac{1}{16} \psi(c,b,a,d)-\frac{7}{32} \psi(d,a,b,c)-$$$$\frac{3}{32} \psi(d,a,c,b)-\frac{3}{16} \psi(d,b,a,c)+\frac{3}{16} \psi(d,c,b,a).$$

For the identity $(\ref{newident431})$ we have 

$$f(a;b,(cd)e)-f(a;b,c(de))=$$


$$-\frac{1}{16}\psi (a,d,e,c) b-\frac{3}{32} \psi (b,e,c,d)a- \frac{3}{16} \psi (c,d,b,e) a-$$ $$\frac{3}{8}  \psi (c,d,e,b) a-\frac{3}{32} \psi (c,e,b,d) a-\frac{3}{32} \psi (c,e,d,b) a-$$ $$\frac{1}{32} \psi (a,e,d,c) b-\frac{3}{32} \psi (d,e,a,c) b-\frac{3}{32} \psi (d,e,c,a) b+$$ $$\frac{1}{32} \psi (a,e,c,d) b-\frac{1}{12} \psi (b,c,e,a d)+\frac{3}{16} \psi (b,d,e,c) a+$$ $$\frac{1}{12} \psi (b,e,d,a c)+\frac{3}{32} \psi (b,e,d,c) a+\frac{1}{16} \psi (c,d,a,e) b+$$ $$\frac{1}{16} \psi (c,d,a e,b)-\frac{1}{16} \psi (c,d,b,a e)-\frac{3}{16} \psi (c,d,b e,a)+$$ $$\frac{1}{8}  \psi (c,d,e,a) b+\frac{1}{12} \psi (c,d,e,a b)+\frac{1}{32} \psi (c,e,a,d) b+$$ $$\frac{1}{12} \psi (c,e,b,a d)+\frac{1}{32} \psi (c,e,d,a) b-\frac{1}{12} \psi (c,e,d,a b)-$$ $$\frac{1}{16} \psi (d,e,a c,b)-\frac{1}{48} \psi (d,e,b,a c)+\frac{9}{32} \psi (d,e,b,c) a+$$ $$\frac{3}{16} \psi (d,e,b c,a)+\frac{9}{32} \psi (d,e,c,b) a.$$

\end{proof}

Let us define the set $B_1(X)=X,$ $B_2(X)=\{x_{i_1}x_{i_2}| {i_1}\leq {i_{2}}, \text{ and } x_{i_1},x_{i_2}\in X\},$ and for $n\geq 3$

$$B_n(X)=\{f(x_{i_1};x_{i_2},x_{i_3}\cdots x_{i_{n}})|  {i_2}\leq\cdots \leq {i_{n}}, \text{ and } x_{i_1},x_{i_2},\cdots ,x_{i_{n}}\in X\}.$$

\begin{lemma}\label{main lemma} Every  left-normed element $w\in J(X)$ of degree $n\geq3,$ can be expressed as a linear combination of elements in $B_n(X).$
\end{lemma}

\begin{proof}
We prove the statement by induction on degree $n.$
The basis of induction is $n = 3$ follows from identity $(\ref{newident41}).$

Assume that the statement of Lemma is true for elements whose degree are less than $n>3.$  Let $u\in J(X),$ and $u=wx,$ where $x\in X$ and $w$ is the element of  $J(X)$ whose degree less than $n.$ Therefore, by induction on $n$ we may assume that $w=f(x_{i_1};x_{i_2},x_{i_3}\cdots x_{i_{k}})\in B_k(X),$  where $k<n.$

Then we consider  $$u=wx=f(x_{i_1};x_{i_2},x_{i_3}\cdots x_{i_{k}})x=$$
(by identity $(\ref{newident44})$)
$$\frac{1}{2}f(x_{i_1};x_{i_2},x_{i_3}\cdots x_{i_{k}}x)+\frac{1}{2}f(x;x_{i_1},x_{i_2}x_{i_3}\cdots x_{i_{k}}),$$
and by the identities $(\ref{newident43})$ and $(\ref{newident431})$ each of the above elements in $B_n(X).$

\end{proof}

Let $J_n(X)$
be the $n$-th homogeneous part of $J(X).$

\begin{corollary}
The set $B_n(X)$ forms a linear basis of $J_n(X).$
\end{corollary}
\begin{Proof}
Follows from Lemma $\ref{main lemma}$ and from Theorem $\ref{Jordan elements}.$
\end{Proof}

\begin{proof}[Proof of Theorem $\ref{free J(X) special}$]
It is sufficient to show that algebras $J(X)$ and $SJ(X)$ are isomorphic. Let $\phi$ be the natural homomorphism from $J(X)$ to $SJ(X)$ sending $x$ to $x,$ for $x\in X.$ 
By Lemma $\ref{main lemma}$ the vector space $J_n(X)$ is spanned by the set $B_n(X).$ We note that the number of elements in $B_n(X)$ of degree $n$ is equal to the number of elements of perm algebra $P(X)$ of degree $n.$  Suppose  $\phi$ is not injective then there exists a multi-linear polynomial of degree $n$ in K$er\phi$.  Then we have a linear combination of elements  $f(x_{i_1};x_{i_2},\ldots,x_{i_n})$ which is zero in $SJ(X).$ It contradicts to Theorem $\ref{Jordan elements}.$ Therefore, K$er\phi =(0).$
\end{proof}

\newpage

\begin{center} ACKNOWLEDGMENTS   \end{center}
This research was funded by the Science Committee of the Ministry of Education and Science of the Republic of Kazakhstan (Grant No. AP14870282) and by  the grant ``Táýelsizdik urpaqtary'' of the Ministry of Information and Social Development of the Republic of Kazakhstan.
The authors are grateful to Professor A. Dzhumadil'daev for his stimulating questions about the identities of perm algebras under commutators. The authors are also grateful to Professor P. Kolesnikov for his comments and advice, thanks to which the article has acquired such a full-fledged look.

\end{document}